\documentclass[12pt]{amsart}

\usepackage{amsmath,amsfonts,amsthm}

\renewcommand{\theequation}{\thesection.\arabic{equation}}
\newtheorem{theorem}{Theorem}
\newtheorem{lemma}{Lemma}
\newtheorem{proposition}{Proposition}

\newtheorem{remark}{Remark}
\newtheorem{definition}{Definition}

\newcommand{\eqnsection}{
\renewcommand{\theequation}{\thesection.\arabic{equation}}
    \makeatletter
    \csname  @addtoreset\endcsname{equation}{section}
    \makeatother}
\eqnsection

\def\r{{\mathbb R}}
\def\e{{\mathbb E}}
\def\p{{\mathbb P}}
\def\z{{\mathbb Z}}
\def\n{{\mathbb N}}

\def\ee{\mathrm{e}}
\def\d{\, \mathrm{d}}
\def\l{\ell}

\def\Z{{\mathbb Z}}

\def\N{{\mathbb N}}
\def\E{{\mathbb E}}

\def\indic{{\bf 1}}

\author[O. Zindy]{Olivier ZINDY}
\address{Weierstrass Institute for Applied Analysis and Stochastics,
Mohrenstrasse 39, 10117 Berlin, Germany}
\email{zindy@wias-berlin.de}

\keywords{directed trap model, random walk, scaling limit, subordinator, aging} \subjclass[2000]{primary 60K37, 60G52, 60F17;
secondary 82D30}

\title[Directed trap models]{Scaling limit and aging for directed trap models}

\begin{document}

\maketitle
\bigskip

{\footnotesize \noindent{\slshape\bfseries Abstract.} }
 We consider one-dimensional directed trap models and suppose that the trapping times are heavy-tailed.
 We obtain the inverse of a stable subordinator as scaling limit and prove an aging phenomenon expressed in terms of the generalized arcsine law. These results confirm the universality of this phenomenon described by Ben Arous and \v Cern\'y for a large class of graphs.
 \bigskip
\bigskip

\section{Introduction}

What is usually called {\it aging} is a dynamical out-of-equilibrium physical phenomenon observed in disordered systems like spin-glasses at low temperature. It is defined by the existence of a limit of a given two-time (usually denoted by $t_\omega$ and $t_\omega+t$) correlation function of the system as both times diverge keeping a fixed ratio between them. The limit should be a non-trivial function of the ratio. It has been extensively studied in the physics literature, see \cite{bouchaud-cugliandolo-kurchan-mezard} and therein references.

The {\it trap model} is a model of random walk that was first  proposed by Bouchaud and Dean \cite{bouchaud, bouchaud-dean} as a toy model for studying this aging phenomenon.
In the mathematics litterature, much attention has recently been given to the trap model, and many aging result were derived from it. The trap model on $\z$ is treated in \cite{fontes-isopi-newman} and \cite{benarous-cerny05},  on $\z^2$ in \cite{benarous-cerny-mountford}, on $\z^d$ $(d \ge 3)$ in \cite{benarous-cerny07b} and on the hypercube in  \cite{benarous-bovier-gayrard03a,benarous-bovier-gayrard03b}. A comprehensive approach to obtaining aging results for the trap model in various settings was later developed in \cite{benarous-cerny07a}. The striking fact is that these aging results are identical for $\z^d,$ $d\ge 2$ and the large complete graph, or the REM. In other terms, the mean-field results are valid from infinite dimension down to dimension $2.$

The one-dimensional trap model has some specific features that distinguish it from all other cases. The most useful feature is that we can identify its scaling limit as an interesting one-dimensional singular diffusion in random environment, see \cite{fontes-isopi-newman}. This process differs considerably from the scaling limit for $d\ge2,$ namely the {\it fractional kinetics} process, i.e. the time change of a $d$-dimensional Brownian motion by the inverse of an independent $\alpha$-stable subordinator, see \cite{benarous-cerny07b}.  In fact, the universality of the aging phenomenon is a question about the transient part of relaxation to equilibrium and not necessarily related to equilibrium questions.

Here, we give an answer to a question of Ben Arous and \v Cern\'y  \cite{benarous-cerny06a} by studying the influence of a drift in the one-dimensional trap model.
We identify the scaling limit of the so-called {\it directed trap model} with the inverse of an $\alpha$-stable subordinator and prove an aging result  expressed  in terms of the generalized arcsine law. These results confirm the universality of the phenomenon described by Ben Arous and \v Cern\'y \cite{benarous-cerny07a}. Furthermore, this extends some results of Monthus \cite{monthus}, who studies the influence of a bias in the high disorder limit (i.e. when $\alpha$ tends to zero with our notations, see (\ref{ass:stable})) using renormalization arguments. Note that the ideas of the proof developed in this paper are deduced from a strong comparison with one-dimensional random walks in random environment in the sub-ballistic regime. Indeed, analogous results are obtained for this asymptotically equivalent model in \cite{enriquez-sabot-zindy-1} (using \cite{enriquez-sabot-zindy-2}) and \cite{enriquez-sabot-zindy-3}.

The rest of the paper is organized as follows. The main results
are stated in Section \ref{s:result}. In Section \ref{s:prelim},
we present some elementary results about the environment, the
embedded random walk as well as preliminary estimates, which will
be frequently used throughout the paper. Section \ref{s:scaling}
and Section \ref{s:aging} are respectively devoted to the proof of
the scaling limit and to the proof of the aging result.

\section{Notations and main results} \label{s:result}

  Let us first fix $0<\varepsilon \le1/2.$ Then, the {\it directed trap model} is the nearest-neighbour continuous-time Markov process $X=(X_t)_{t \ge 0} $ with state space $\z,$ given by $X_0=0$ and with jump rates
 \begin{eqnarray}
c(x,y) :=\left\{\begin{array}{lll} \left({1 \over 2} + \varepsilon\right) \tau_x^{-1} & {\rm if}
\ y=x+1,
\\
 \left({1 \over 2}- \varepsilon\right)  \tau_x^{-1}  &  {\rm if} \ y=x-1,

\end{array}
\right.
 \end{eqnarray}
and zero otherwise, where ${\tau}=(\tau_x)_{x \in \z}$ is a family of positive i.i.d. heavy-tailed random variables. More precisely, we suppose that there exists $\alpha \in (0,1)$ such that
 \begin{eqnarray}
 \label{ass:stable}
  \lim_{u \to \infty} u^{\alpha} \,  \p (\tau_x \ge u)=1.
 \end{eqnarray}
In particular, this implies $\e \left[\tau_x\right]=+\infty.$ Sometimes $\tau$ is called random environment of traps. The Markov process $X_t$ spends at site $x$ an exponentially distributed time of mean $\tau_x$, and then jumps to the right with probability $p_\varepsilon:=({1 \over 2} + \varepsilon)$ and to the left with probability $q_\varepsilon:=({1 \over 2} - \varepsilon).$ Therefore, $X$ is a time change of a  discrete-time biased random walk on $\z.$ More precisely, we define the {\it clock process} and the {\it embedded random walk} associated with $X$ as follows.
\begin{definition}
\label{d:clock+embedded}
Let $S(0):=0$ and let $S(k)$ be the time of the $k$-th jump of $X,$ for $k \in \n^*.$ For $s \in \r_+,$ we define $S(s):=S(\lfloor s \rfloor)$ and call $S$ the {\it clock process.} Define the embedded discrete-time random walk $(Y_n)_{n \ge 0}$ by $Y_n:=X_t$ for $S(n)\le t < S(n+1).$ Then obviously, $(Y_n)_{n \ge 0}$ is a biased random walk on $\z.$
\end{definition}

Observe that $(Y_n)_{n \ge 0}$ satisfies
$
 \p(Y_{n+1}=Y_n+1)={1 \over 2} + \varepsilon=1- \p(Y_{n+1}=Y_n-1),
$
for all $n \ge 0.$ Therefore, $(Y_n)_{n \ge 0}$ is transient to $+ \infty$ and the law of large numbers implies that, $\p$-almost surely,
 \begin{eqnarray}
 \label{eq:LLN}
{Y_n \over n}  \longrightarrow v_\varepsilon:=2\varepsilon>0, \qquad n \to \infty.
 \end{eqnarray}
Furthermore, it follows from the definition of $X$ that the clock process can be written
 \begin{eqnarray}
 \label{def:clock}
 S(k)=\sum_{i=0}^{k-1} \tau_{Y_i} {\bf e}_i, \qquad k \ge 1,
 \end{eqnarray}
where $({\bf e}_i)_{i\ge0}$ is a family of i.i.d. mean-one exponentially distributed random variables. We always suppose that the ${\bf e}_i$'s are defined in this way. Then, the process $(X_t)_{t \ge 0}$ satisfies
 \begin{eqnarray}
 \label{def:X}
 X_t=Y_{S^{-1}(t)}, \qquad \forall \,  t \in \r_+,
 \end{eqnarray}
where the right-continuous inverse of an increasing function $\phi$ is defined by
$
\phi^{-1}(t):= \inf \{u \ge 0: \, \phi(u)> t\}.
$

 Now, let us fix $T>0$ and denote by $D(\left[0,T\right])$ the space of c\`adl\`ag functions from $\left[0,T\right]$ to $\r.$ Moreover, let $X^{(N)}$ be the sequence of elements of $D(\left[0,T\right])$ defined by
 \begin{equation}
X^{(N)}_t:= {X_{t N} \over N^{\alpha}}, \qquad 0 \le t \le T.
\end{equation}
Then, the scaling limit result can be stated as follows.
\begin{theorem}
\label{t:scaling}  The distribution of the process $(X^{(N)}_t; \, 0 \le t \le T)$ converges weakly to the distribution of $(v^{\#}_\varepsilon \, V_\alpha^{-1}(t); \, 0 \le t \le T)$ on $D(\left[0,T\right])$ equipped with the uniform topology, where $(V_\alpha(t); \, t \ge 0)$ is an $\alpha$-stable subordinator satisfying $\e[\ee^{-\lambda V_\alpha(t)}] =\ee^{-t \lambda^\alpha},$ and $v^{\#}_\varepsilon:={   \sin( \alpha\pi)  \over\alpha \pi  } v_\varepsilon^{\alpha}  ={   \sin( \alpha\pi)  \over\alpha \pi  } (2\varepsilon)^{\alpha} .$
\end{theorem}
Although this result can be compared with the limit in \cite{benarous-cerny07b}, we do not obtain the {\it fractional kinetics} process. This difference can be explained by recalling that the fractional kinetics process is the time change of a Brownian motion by the inverse of an independent $\alpha$-stable subordinator while our embedded random walk satisfies the law of large numbers with positive speed, see (\ref{eq:LLN}).
 Furthermore, observe that the case $\varepsilon= 1/2$ is trivial; indeed $Y$ is deterministic, $v_\varepsilon =1$ and the clock process, which can be written $S(k)=\sum_{i=0}^{k-1} \tau_{i} {\bf e}_i,$ is just a sum of i.i.d. heavy-tailed random variables. Now let us state the second main result concerning the aging phenomenon.
\begin{theorem}
\label{t:aging}
For all $h>1,$ we have
\begin{equation}
\lim_{t\to\infty} \p(X_{th}=X_t)= {\sin(\alpha\pi)\over \pi} \int_{0}^{1/h} y^{\alpha-1} (1-y)^{-\alpha}\d y.
\end{equation}
\end{theorem}

\begin{remark}
\label{r:fonctioncorrelation}  Note that, in \cite{bertin-bouchaud}, Bertin and Bouchaud study the average position of the random walk at time $t_\omega+t$ given that a small bias $h$ is applied at time $t_\omega.$ They found several scaling regime depending on the relative value of $t,$ $t_\omega$ and $h.$
  \end{remark}
In the following, $C$ denotes a constant large enough, whose value can change from line to line.

\section{Preliminary estimates}  \label{s:prelim}
In this section, we list some properties of the environment $\tau$ and of the embedded random walk $Y$ as well as preliminary results.

 \subsection{The environment}
Let us define the critical depth for the first $n$ traps of the environment by
\begin{equation}
g(n):={n^{1/\alpha} \over (\log n)^{2 \over 1-\alpha}}.
\end{equation}
Then, we can introduce the notion of {\it deep traps} as follows:
\begin{eqnarray}
\label{eq:deep1}
 \delta_1&=&\delta_1(n):= \inf \{x \ge 0: \; \tau_x \ge g(n)\},
 \\
 \label{eq:deep2}
 \delta_{j}&=&\delta_{j}(n):= \inf \{x > \delta_{j-1}: \; \tau_x \ge g(n)\},     \qquad j \ge 2.
\end{eqnarray}
The number of such deep traps before site $n$ will be denoted by $\theta_n$ and defined by
\begin{equation}
\theta_n:= \sup\{j \ge 0: \; \delta_{j} \le n \},
\end{equation}
where $\delta_0:=0.$ Now, let us define
\begin{equation}
\varphi(n):=\p(\tau_1 \ge g(n)),
\end{equation}
and observe that (\ref{ass:stable}) implies $\varphi(n) \sim g(n)^{-\alpha} ,$ $n \to \infty.$
We introduce now the following series of events, which will occur with high probability, when $n$ goes to infinity:
\begin{eqnarray}
 \mathcal{E}_1(n)&:=&\left\{ n \varphi(n) \Big(1-{ 1 \over \log n}\Big) \le  \theta_n \le n \varphi(n) \Big(1+{ 1 \over \log n}\Big)\right\},
 \\
 \mathcal{E}_2(n)&:=&\left\{\delta_1 \wedge \min_{1\le j \le \theta_n-1} (\delta_{j+1}-\delta_j)  \ge \rho(n) \right\},
 \\
 \mathcal{E}_3(n)&:=&\left\{ \max_{-\nu(n)  \le x \le 0} \tau_x  < g(n) \right\},
\end{eqnarray}
where $\rho(n)$ and $\nu(n)$ are given, for some $0<\kappa<1/3$ and $0<\gamma<1,$ by 
\begin{eqnarray}
\rho(n)&:=& n^\kappa,
 \\
\nu(n)&:=&\lfloor (\log n)^{1+\gamma}\rfloor.
\label{eq:defnu}
\end{eqnarray}
 In words, $\mathcal{E}_1(n)$ requires that the number of deep traps is not too large, $\mathcal{E}_2(n)$ requires that the distance between two deep traps is large enough and $\mathcal{E}_3(n)$ ensures that the time spent by $X$ on $\z_-$ is negligible.
\begin{lemma}
\label{l:preliminaries}
Let $\mathcal{E}(n):= \mathcal{E}_1(n)\cap \mathcal{E}_2(n)\cap \mathcal{E}_3(n)$, then we have
\begin{equation}
\lim_{n\to\infty} \p(\mathcal{E}(n))=1.
\end{equation}
\end{lemma}
\begin{proof} Note that the number of traps deeper
than $g(n)$ in the first $n$ traps is a binomial random variable with
parameter $(n, \varphi(n)).$  Then, recalling (\ref{ass:stable}), the proof of Lemma \ref{l:preliminaries} is easy and left to the reader. \end{proof}

Since we want to consider disjoint intervals of size $2 \nu(n)$ around the $\delta_j$'s, we introduce now a subsequence of the deep traps defined above (see (\ref{eq:deep1})-(\ref{eq:deep2})). These so-called {\it$*$-deep traps} are defined as follows:
\begin{eqnarray}
 \delta_1^*&=&\delta_1^*(n):= \inf \{x \ge \nu(n): \; \tau_x \ge g(n)\},
 \\
 \delta_{j}^*&=&\delta_{j}^*(n):= \inf \{x > \delta_{j-1}^*+2 \nu(n): \; \tau_x \ge g(n)\},     \qquad j \ge 2.
\end{eqnarray}
The number of such $*$-deep traps before site $n$ will be denoted by $\theta_n^*$ and defined by
\begin{equation}
\theta_n^*:= \sup\{j \ge 0: \; \delta_{j}^* \le n \}.
\end{equation}
For any $\nu \in \n^*$ and any $x \in \z,$ let us denote by $B_\nu(x)$ the interval $[x-\nu,x+\nu].$ Observe that the intervals $(B_{\nu(n)}(\delta_j^*))_{1 \le j \le \theta_n^*}$ will be made of independent and identically distributed portions of environment $\tau$ (up to some translation).

The following lemma tells us that the $*$-deep traps coincide
with the deep traps with an overwhelming probability
when $n$ goes to infinity.
\begin{lemma}
\label{l:*}
If $\mathcal{E}^*(n):= \{\theta_n=\theta_n^*\}$, then we have
\begin{equation}
\lim_{n\to\infty} \p(\mathcal{E}^*(n))=1.
\end{equation}
\end{lemma}
\begin{proof} Recall first that the $*$-deep traps constitute a subsequence of the deep traps. Furthermore, we have $\mathcal{E}_2(n) \subset \mathcal{E}^*(n)$ for all large $n.$ Therefore, Lemma \ref{l:preliminaries} implies Lemma \ref{l:*}.
\end{proof}

\subsection{The embedded random walk}
Let us first introduce
 the hitting time $\zeta_n$ of site $n$ for the embedded random walk $Y$ defined by
 \begin{equation}
\zeta_n:=\inf\{k \ge 0: \, Y_k=n\}, \qquad n \in \n.
\end{equation}
 Since $Y$ is transient to $+ \infty$, we have $\zeta_n<\infty,$ for all $n \ge 0$ almost surely.
To control the behavior of $Y$, let us state the following result.
\begin{lemma}
\label{f:preliminaries-traj}
Let $\mathcal{A}(n):=\{ \min_{0\le i < j \le \zeta_n} (Y_j-Y_i)  > - \nu(n)\}$, then we have
\begin{equation}
\lim_{n\to\infty} \p(\mathcal{A}(n))=1.
\end{equation}
\end{lemma}
Observe that, on $\mathcal{A}(n),$ each time $X$ (or $Y$) hits a site $x,$ it will necessarily exit  $B_{\nu(n)}(x)$ on the right.
\begin{proof} Let us fix $c>v_{\varepsilon}^{-1}.$ Then, observe that the law of large numbers implies that $\p(\zeta_n\le cn)\to1,$ $n \to \infty.$ Therefore, it is sufficient to prove that $\p(\min_{0\le i < j \le cn} (Y_j-Y_i)  \le - \nu(n))\to 0,$ $n \to \infty.$ Let us write
\begin{eqnarray}
\p\Big(\min_{0\le i < j \le cn} (Y_j-Y_i)  \le - \nu(n)\Big) &\le& C \, n^2 \max_{0\le  j \le cn} \p(Y_j  \le - \nu(n)). 
\label{eq:ld11}
\end{eqnarray}
Now,  for all $x$ and every $t \le 0,$ an application of Chebycheff's inequality yields
\begin{eqnarray}
\p(Y_j / j  \le x) &=&  \E [{\bf 1}_{\{ Y_j-jx\le0 \}}] \le \E [\ee^{t(Y_j-jx)}]
\nonumber
\\
&=& \ee^{-j t x} \E [\ee^{t Y_1}]^j=\ee^{-j  \{ tx - \Lambda(t) \}},
\label{eq:ld12}
\end{eqnarray}
where $\Lambda(t):=\log \e[\ee^{t Y_1}]$ denotes the logarithmic moment generating function associated with the law of $Y_1.$ By taking the infimum over $t\le0$ in (\ref{eq:ld12}), we get 
\begin{eqnarray}
\p(Y_j / j  \le x) &\le& \ee^{-j I(x)},
\label{eq:ldupperbound}
\end{eqnarray}
where $I(x):=\sup_{t \le 0}\{tx-\Lambda(t)\}.$ Note that (\ref{eq:ldupperbound}) corresponds to the upper bound in the LDP (large deviation principle) for an i.i.d. sequence (see Cramer's theorem in $\r,$ \cite{dembo-zeitouni98} page 27). Since $\e[Y_1]=v_\varepsilon>0,$ we have $I(x)=\sup_{t \in \r}\{tx-\Lambda(t)\}$ for $x \le v_\varepsilon$ (see (2.2.7) in \cite{dembo-zeitouni98}), which means that $I$ is the convex rate function associated with $Y.$ Now, assembling (\ref{eq:ld11}) and (\ref{eq:ldupperbound}) yields
\begin{eqnarray}
\p\Big(\min_{0\le i < j \le cn} (Y_j-Y_i)  \le - \nu(n)\Big)  &\le& C \, n^2 \max_{0\le  j \le cn} \ee^{-j I\left({- \nu(n) \over j}\right)}.
\label{eq:ld1}
\end{eqnarray}
Then, Lemma \ref{f:preliminaries-traj} will be a consequence of 
 \begin{equation} \label{eq:ld2}
\sup_{x \le 0} {I(x) \over x} \le \log r_\varepsilon <0,
\end{equation}
where  $r_\varepsilon:= q_\varepsilon/ p_\varepsilon<1.$
To prove (\ref{eq:ld2}),
 observe that an easy computation yields $\Lambda(\log r_\varepsilon)=0.$  Therefore, by definition $I(x)\ge x \log r_\varepsilon$ for all $x \le 0,$ which gives  (\ref{eq:ld2}).
 
Finally, assembling (\ref{eq:ld1}) and  (\ref{eq:ld2}) implies that $\p(\min_{0\le i < j \le cn} (Y_j-Y_i)  \le - \nu(n)) \le C n^2 \ee^{\nu(n) \log r_\varepsilon}$ which tends to 0 when $n$ tends to infinity (recall that $\nu(n)$ is defined in (\ref{eq:defnu}) and satisfies $\nu(n)=\lfloor (\log n)^{1+\gamma}\rfloor$ for some $0<\gamma<1$).
\end{proof}

\subsection{Between deep traps} Here, we prove that the time spent between deep traps is negligible.
\begin{lemma}
\label{l:IA}  Let us define $ \mathcal{I}(n):=\left\{ \sum_{i=0}^{\zeta_n} \tau_{Y_i} {\bf e}_i {\bf 1}_{\{\tau_{Y_i} < g(n)\}}  <  { n^{1/\alpha} \over \log n }   \right\}.$ Then, we have
  \begin{equation}
\p\left(\mathcal{I}(n)\right) \to 1, \qquad n \to \infty.
\end{equation}
\end{lemma}
\begin{proof}  Observe first that, on $ \mathcal{A}(n),$ we have $\inf_{i \le \zeta_n} Y_i \ge -\nu(n)$  and that Lemma \ref{f:preliminaries-traj}  implies $\p\left(\mathcal{I}(n)^c\right)=\p\left(\mathcal{I}(n)^c \cap  \mathcal{A}(n) \right)+o(1).$  Therefore, using Markov inequality, we only have to prove that
  \begin{equation} \label{sumgreen}
\e\bigg[ \sum_{i=0}^{\zeta_n} \tau_{Y_i} {\bf e}_i{\bf 1}_{\{{Y_i} \ge -\nu(n)\}}   {\bf 1}_{\{ \tau_{Y_i} < g(n)\}}  \bigg] = o\Big(  { n^{1/\alpha} \over \log n } \Big), \qquad n \to \infty.
\end{equation}
After reaching $x \in\left[ -\nu(n) , n \right]$ (if $x$ is reached), the process $Y$ visits $x$ a geometrically distributed number of times before hitting $n.$  The parameter of this geometrical random variable is equal to $q_\varepsilon+p_\varepsilon\, \psi(x,n),$ where $\psi(x,n)$ denotes the probability that $Y$ starting at $x+1$ hits $x$ before $n.$ An easy computation yields that
  \begin{equation}
  \psi(x,n)=r_\varepsilon{1- r_\varepsilon^{n-x-1} \over 1-r_\varepsilon^{n-x}},
  \end{equation}
 with  $r_\varepsilon= q_\varepsilon/ p_\varepsilon<1.$
 We will denote by $G(x,n)$ the mean of this geometrical random variable. Moreover, let us use respectively  $\p_{\tau}(\cdot)$ and $\e_{\tau}[\cdot]$ to denote the conditional probability and the conditional expectation with respect to $\tau$ (sometimes called quenched expectation).
Recalling that each visit takes an exponential time of mean $\tau_x,$ we obtain
  \begin{equation} \label{sumgreen2}
\e_\tau \bigg[ \sum_{i=0}^{\zeta_n} \tau_{Y_i} {\bf e}_i{\bf 1}_{\{{Y_i} \ge -\nu(n)\}}   {\bf 1}_{\{ \tau_{Y_i} < g(n)\}}  \bigg]  \le \sum_{x=-\nu(n)}^{n} \tau_x (1+G(x,n))  {\bf 1}_{\{\tau_{x} < g(n)\}}.
 \end{equation}
Since $x \mapsto G(x,n)$ is decreasing and $G(-\nu(n),n) \to (1-v_\varepsilon)/v_\varepsilon,$ when $n \to \infty,$ we get that the expectation in (\ref{sumgreen}) is, for all large $n,$ less than
 $C n \, \e\left[\tau_0 \, ; \, \tau_0 <g(n) \right]=C n \, \e\left[\tau_0 \, ; \, 1<\tau_0 <g(n) \right]+O(n).$ Now, let us fix $0<\rho<1$ and introduce
$\omega=\omega(n):=\inf \{ j\ge 0: \;  \rho \le \rho^{j} g(n) <1\}.$ Then, we get
 \begin{eqnarray}
  \e\left[\tau_0 \, ; \, 1<\tau_0 <g(n) \right] &\le&  g(n) \sum_{j=0}^{\omega-1}  \rho^{j} \p( \tau_0> \rho^{j+1} g(n))
\\
&\le& C  g(n)^{1-\alpha} \sum_{j=0}^{\omega-1}  \rho^{-\alpha j} \le C g(n)^{1-\alpha},
\nonumber
\end{eqnarray}
where we used the fact that (\ref{ass:stable}) yields that there exists $0<C<\infty$ such that $ \p (\tau_x \ge u)\le C u^{-\alpha},$ for all $u > 0.$ Therefore, recalling (\ref{sumgreen}), the fact that $n g(n)^{1-\alpha}$ is a $o(n^{1/\alpha}/ \log n)$ concludes the proof of Lemma \ref{l:IA}.  \end{proof}

\subsection{Occupation time of a deep trap}
Since $\zeta_y<\infty$ for all $y \in \n,$ we can properly define for $x \in \n,$
 \begin{eqnarray}
 \label{eq:defTx}
 T_x=T_x(n)&:=&\sum_{0}^{\zeta_{x+\nu(n)}} \tau_{Y_i} {\bf e}_i{\bf 1}_{\{{Y_i}=x\}},
 \\
   \overline T_x=\overline T_x(n)&:=&\sum_{0}^{\zeta_{x+\nu(n)}} \tau_{Y_i} {\bf e}_i{\bf 1}_{\{{Y_i}\in B_{\nu(n)}(x)\}}.
\end{eqnarray}
Moreover, let us introduce $\p^x$ and $\e^x$ the probability and the expectation associated with the process starting at site $x.$ We have the following estimate for the Laplace transforms of $T_x$ and $\overline T_x.$
\begin{lemma}
\label{l:laplacetransform} For all $x \in \n$ and all $\lambda>0,$ we have
 \begin{equation}
\label{eq:laplace} \E^x \Big[1-e^{-\lambda_n {T_x}} \vert \tau_x \ge g(n) \Big] \sim
{\p(\tau_x \ge g(n))^{-1} \over n} { \alpha \pi \over \sin( \alpha\pi)} \, v_\varepsilon^{-\alpha}
\, \lambda^\alpha, \qquad n\to \infty,
\end{equation}
where $\lambda_n:=\lambda / n^{1/\alpha}.$ Moreover, the same result holds with $T_x$ replaced by $\overline T_x.$
\end{lemma}
\begin{proof}
Let us first write
 \begin{equation}
 \label{eq:laplaceoverlineT}
 \E^x\Big[(1-e^{-\lambda_n {T_x}}){\bf 1}_{\{\tau_x \ge g(n)\}}  \Big]=  \E\Big[ \E^x_\tau[1-e^{-\lambda_n {T_x}}   ]  {\bf 1}_{\{\tau_x \ge g(n)\}}  \Big].
  \end{equation}
Starting at site $x,$ the process $Y$ visits $x$ a geometrically distributed number of times before reaching $x+\nu(n).$ An easy computation yields that  the mean of this geometrical variable, denoted by $G(x,x+\nu(n))$ satisfies $1+G(x,x+\nu(n)) \to v_\varepsilon^{-1},$ when $n \to \infty.$ Therefore, recalling that each visit takes an exponential time of mean $\tau_x,$ we obtain
 \begin{equation}
 \E^x_\tau[e^{-\lambda_n{T_x}}   ]=  {1 \over 1 +\lambda_n  v_\varepsilon^{-1}\tau_x } +o(n^{-1/\alpha}), \qquad n \to \infty.
  \end{equation}
Now, using an integration by part, we get that $ \E^x\Big[(1-e^{-\lambda_n {T_x}}){\bf 1}_{\{\tau_x \ge g(n)\}}  \Big]$ is equal to
 \begin{equation}
\Big[- {\lambda_n v_\varepsilon^{-1} z  \over 1+\lambda_n  v_\varepsilon^{-1}  z} \p(\tau_x \ge z)
\Big]_{g(n)}^\infty +\int_{g(n)}^\infty
{\lambda_n v_\varepsilon^{-1} \over (1+\lambda_n v_\varepsilon^{-1} z)^2} \p(\tau_x \ge z) \d z  + o(n^{-1/\alpha}).
\end{equation}
The first term is lower than
$C\lambda_ng(n)^{1-\alpha}=C \lambda_n^\alpha (\lambda_n g(n))^{1-\alpha}=o(n^{-1}),$ since $\alpha<1$. For the second term, using (\ref{ass:stable}), we can estimate $\p(\tau_x \ge z)$ by
$
(1-\eta)z^{-\alpha} \le \p(\tau_x \ge z)\le (1+\eta) z^{-\alpha},
$
for any $\eta,$ when $n$ is sufficiently large (recall that $g(n) \to \infty,$ when $n \to \infty$). Hence, we are lead to compute the integral
 \begin{equation}
\int_{g(n)}^\infty {\lambda_n v_\varepsilon^{-1} \over (1+\lambda_n v_\varepsilon^{-1} z)^2}
z^{-\alpha} \d z =(\lambda_n v_\varepsilon^{-1})^\alpha \int_{{\lambda_n v_\varepsilon^{-1}g(n)}\over
1+\lambda_n v_\varepsilon^{-1} g(n)}^1 y^{-\alpha}(1-y)^\alpha \d y,
 \end{equation}
(making the change of variables $y=\lambda_n v_\varepsilon^{-1} z/(1+\lambda_n v_\varepsilon^{-1} z)$). For $\alpha<1$ this integral converges, when $n\to \infty,$ to
$
\Gamma(\alpha+1)\Gamma(-\alpha+1)={\pi \alpha \over \sin(\pi \alpha)},
$
which concludes the proof of (\ref{eq:laplace}).

To prove that the result is true with $\overline T_x$ in place of $T_x,$ observe first that $\p(\tau_x\ge g(n) ; \max_{y \in B_{\nu (n)}(x) \setminus \{x\}} \tau_y \ge g(n))=o(n^{-1}),$ when $n \to \infty,$ which implies
 \begin{equation}
 \label{eq:overline-sans}
\E^x\Big[(1-e^{-\lambda_n {\overline T_x}}){\bf 1}_{\{\tau_x \ge g(n)\}}  \Big]= \E^x\Big[(1-e^{-\lambda_n {\overline T_x}}){\bf 1}_{\mathcal{E}_4(n)}  \Big]+o(n^{-1}),
\end{equation}
where $\mathcal{E}_4(n):= \{\max_{y \in B_{\nu (n)}(x) \setminus \{x\}} \tau_y < g(n) \le \tau_x \}.$ Then, let us introduce $\tilde T_x:= \sum_{0}^{\zeta_{x+\nu(n)}} \tau_{Y_i} {\bf e}_i{\bf 1}_{\{{Y_i}\in B_{\nu(n)}(x)\setminus \{x\} \}}=\overline T_x- T_x$ and write
 \begin{equation}
 \label{eq:diff}
 \E^x\Big[(e^{-\lambda_n {T_x}}-e^{-\lambda_n {\overline T_x}}){\bf 1}_{\mathcal{E}_4(n)}  \Big]
\le \lambda_n  \E^x\Big[{\tilde T_x}{\bf 1}_{\mathcal{E}_4(n)}  \Big],
\end{equation}
where we used the fact that $1-\ee^{-x}\le x,$ for any $x\in \r.$ Using the same arguments as in the proof of Lemma \ref{l:IA}, we can prove that
 \begin{equation}
 \label{eq:tildeT}
\E^x_\tau\Big[{\tilde T_x}{\bf 1}_{\mathcal{E}_4(n)}  \Big] \le  {\bf 1}_{\{\tau_{x} \ge g(n)\}}  \sum_{y \in B_{\nu(n)}(x)\setminus \{x\}} \tau_y (1+G(y,x+\nu(n))  {\bf 1}_{\{\tau_{y} < g(n)\}}.
\end{equation}
Using the fact that the previous sum depends only on sites $y$ in $B_{\nu(n)}(x)$ which are different from $x,$ together with the same arguments as in the proof of Lemma \ref{l:IA}, we get
$\E^x\Big[{\tilde T_x}{\bf 1}_{\mathcal{E}_4(n)}  \Big] \le C \nu(n) g(n)^{1-\alpha}  \p(\tau_{x} \ge g(n)) \le C  \nu(n) g(n)^{1-2 \alpha}.$ Therefore, we obtain that the left-hand term in (\ref{eq:diff}) is a $o(n^{-1}),$ which together with (\ref{eq:overline-sans}) concludes the proof of Lemma \ref{l:laplacetransform}.
\end{proof}

\section{Proof of Theorem \ref{t:scaling}}  \label{s:scaling}
Let us first define $H_x:=\inf\{ t\ge 0: X_t=x\},$ for any $x\in \n.$
Now, fix $T>0,$ and let $H^{(N)}_t$ be the sequence of elements of $D(\left[0,T\right])$ defined by
 \begin{equation}
H^{(N)}_t:= {H_{\lfloor t N \rfloor} \over N^{1/\alpha}}, \qquad 0 \le t \le T.
\end{equation}

\begin{proposition}
\label{p:scaling limit}  The distribution of the process $( H^{(N)}_t; \, 0 \le t \le T)$ converges weakly to the distribution of $(v_\varepsilon^{\#})^{-1/\alpha} \, V_\alpha(t); \, 0 \le t \le T)$ on $D(\left[0,T\right])$ equipped with the Skorokhod $M_1$-topology, where $(V_\alpha(t); \, t \ge 0)$ is is an $\alpha$-stable subordinator satisfying $\e[\ee^{-\lambda V_\alpha(t)}] =\ee^{-t \lambda^\alpha}.$ 
\end{proposition}
The so called Skorokhod $M_1$-topology is not so common in the literature. Therefore, we refer to \cite{whitt} for detailed account on  $M_1$-topology.

\begin{proof}
Let $0=u_0< u_1< \dots < u_K\le T$ and $\beta_i>0$ for $i \in \{1,\dots , K \}.$ We will check the convergence of the finite-dimensional distributions of $H$ by proving the convergence of $ \E[\exp\{-\sum_{i=1}^K  \beta_i (H^{(N)}_{u_i}- H^{(N)}_{u_{i-1}} )\}].$

 Observe first that for any $u\in\z$ we have  $\p(\max_{y \in B_{\nu(TN)}(u)} \tau_y>g(TN))=o(1),$ when $N \to \infty.$ Then, this remark applied at $u':=\lfloor u_{K-1}N \rfloor-\nu(TN)$ with Lemma \ref{l:IA} yield
\begin{equation}
\label{eq:uknegli}
\p\Big(\sum_{i=0}^{\zeta_{\lfloor u_{K} N  \rfloor }} \tau_{Y_i} {\bf e}_i{\bf 1}_{\{{Y_i} \in B_{\nu(TN)}(u')\}} < C N^{1/\alpha} (\log N)^{-1}  \Big) \to 1, \qquad N \to \infty,
\end{equation}
which means that the time spent by $X$ in $B_{\nu(TN)}(u' )$ is negligible. Recalling that on $\mathcal{A}(TN)$ (whose probability tends to one by Lemma \ref{f:preliminaries-traj}) the process never backtracks more than $\nu(TN),$ we deduce from (\ref{eq:uknegli}) that
\begin{equation}
\label{eq:interhitnegli}
\p\Big( H_{\lfloor u_{K-1}N \rfloor}-H_{u'} < C N^{1/\alpha} (\log N)^{-1}  \Big) \to 1, \qquad N \to \infty.
\end{equation}
Hence, defining $H':= \beta_{K-1} N^{-{1 \over \alpha}} (H_{u'}-H_{ \lfloor u_{K-2}N \rfloor })$ we get 
 \begin{eqnarray} 
\nonumber && \E\left[\ee^{-\sum_{i=1}^{K}  \beta_i (H^{(N)}_{u_i}- H^{(N)}_{u_{i-1}} )} \right] 
\\
\nonumber 
 &=& \E\left[{\bf 1}_{ \mathcal{A}(TN)} \, \ee^{-\sum_{i=1}^{K-2}  \beta_i (H^{(N)}_{u_i}- H^{(N)}_{u_{i-1}} )-  H' } \ee^{- \beta_K (H^{(N)}_{u_{K}}-H^{(N)}_{ u_{K-1} } )}\right]  +o(1)
  \\
 &=& \E\left[\E_\tau \left[ \ee^{-\sum_{i=1}^{K-2}  \beta_i (H^{(N)}_{u_i}- H^{(N)}_{u_{i-1}} )- H' } \right] \e_{\tau,|u'}^{\lfloor u_{K-1}N \rfloor}\left[ \ee^{- \beta_K H^{(N)}_{u_{K}} }\right]\right]  +o(1),
  \label{eq:third}
 \end{eqnarray} 
 where $\e_{\tau,|y}^x$ denotes the law of the process in the
environment $\tau,$ starting at $x$ and reflected at site $y$. The last equality is a consequence of the strong Markov property applied at time $H_{\lfloor u_{K-1}N \rfloor}$ together with the fact that on $\mathcal{A}(TN)$ the process never backtracks more than $\nu(TN).$ Now, observe that the two quenched expectations in (\ref{eq:third}) depend on two disjoint portions of the environment: $(-\infty;  u')\cap \z$ and $[u' , \lfloor u_{K}N \rfloor ) \cap \z.$ Hence, since the $\tau_x$'s are i.i.d., these two quenched expectations are independent random variables and we obtain
 \begin{eqnarray} &&  \E\left[\ee^{-\sum_{i=1}^{K}  \beta_i (H^{(N)}_{u_i}- H^{(N)}_{u_{i-1}} )} \right] 
 \nonumber
 \\
& =&
 \nonumber
\E\left[\ee^{-\sum_{i=1}^{K-2}  \beta_i (H^{(N)}_{u_i}- H^{(N)}_{u_{i-1}} )-  H' } \right] \E\left[  \e_{\tau,|u'}^{\lfloor u_{K-1}N \rfloor} \left[ \ee^{- \beta_K H^{(N)}_{u_{K}} }\right]\right]  +o(1). \end{eqnarray}
Using again (\ref{eq:interhitnegli}) and Lemma \ref{f:preliminaries-traj} we have
 \begin{eqnarray}\nonumber   \E\left[\ee^{-\sum_{i=1}^{K}  \beta_i (H^{(N)}_{u_i}- H^{(N)}_{u_{i-1}} )} \right] 
=
\E\left[\ee^{-\sum_{i=1}^{K-1}  \beta_i (H^{(N)}_{u_i}- H^{(N)}_{u_{i-1}}) } \right] \E^{\lfloor u_{K-1}N \rfloor} \left[ \ee^{- \beta_K H^{(N)}_{u_{K}} }\right]\ +o(1). \end{eqnarray}
By the shift invariance of the environment, it is sufficient to prove that
  \begin{equation}
   \label{eq:onedim}
\e \left[  \ee^{- \beta_K  N^{-1/\alpha} H_{N' }}  \right]  \longrightarrow \exp\Big\{-  { \alpha \pi \over \sin( \alpha\pi)}  v_\varepsilon^{-\alpha} \beta_K^\alpha(u_K-u_{K-1} )\Big\}, \qquad N \to \infty,
 \end{equation}
where $N' := \lfloor u_{K} N \rfloor-\lfloor u_{K-1}N \rfloor \sim (u_K-u_{K-1} )N,$ when $N \to \infty.$ Indeed, iterating this procedure $K-2$ times will give the convergence of the finite-dimensional distributions.

Let us prove (\ref{eq:onedim}). Recalling Lemma \ref{l:preliminaries}, Lemma \ref{f:preliminaries-traj} and Lemma \ref{l:IA}, we obtain
\begin{eqnarray}
  \e \left[  \ee^{- \beta_K  N^{-1/\alpha} H_{N' }}  \right] &=&  \e \left[ {\bf 1}_{\mathcal{E}(N') \cap \mathcal{A}(N') \cap \mathcal{I}(N')}   \ee^{- \beta_K  N^{-1/\alpha} H_{N' }}  \right] +o(1)
  \nonumber
      \\
  &=&   \e \left[   \ee^{- \beta_K  N^{-1/\alpha} \sum_{i=1}^{\theta_{N'}}  T_{\delta_i(N')}  }  \right] +o(1)
    \nonumber
       \\
   &=&  \e \left[ {\bf 1}_{\mathcal{E}^*(N') }  \ee^{- \beta_K  N^{-1/\alpha} \sum_{i=1}^{\theta^*_{N'}}  T_{\delta_i^*(N')}  }  \right] +o(1),
\end{eqnarray}
where $T_x$ is defined in (\ref{eq:defTx}).
Furthermore, since on $\mathcal{E}^*(N') \cap \mathcal{A}(N')$ the process never backtracks before  ${\delta_i^*} -\nu(N')$ after hitting ${\delta_i^*}$ for $1\le i \le\theta^*_{N'},$ we get, by applying successively the strong Markov property at the stopping times $H_{\delta_{\theta_{N'}^*}}, \dots, H_{\delta_{1}^*},$
\begin{eqnarray}
  \e \left[  \ee^{- \beta_K  N^{-1/\alpha} H_{N' }}  \right] &=&  \e \bigg[ {\bf 1}_{\mathcal{E}^*(N')  \cap \mathcal{A}(N') } \prod_{j=1}^{\theta_{N'}^*}   \e_{\tau,|\delta_i^*-\nu}^{\delta_i^*} \left[ \ee^{- \beta_K  N^{-1/\alpha} T_{\delta_i^*}  }  \right]\bigg] +o(1)    \nonumber
      \\
  &\le&    \e \bigg[  \prod_{j=1}^{\underline \theta_{N'}}   \e_{\tau,|\delta_i^*-\nu}^{\delta_i^*} \left[ \ee^{- \beta_K  N^{-1/\alpha} T_{\delta_i^*}  }  \right]\bigg] +o(1),  \end{eqnarray}
where $\underline \theta_{N'}:= N' \varphi({N'}) \big(1-{ 1 \over \log N'}\big).$ Then, observing that the quenched expectations
$(\e_{\tau,|\delta_i^*-\nu}^{\delta_i^*} [ \ee^{- \beta_K  N^{-1/\alpha} T_{\delta_i^*}  }] , \,1 \le j \le \underline{\theta}_{N'} )$ are i.i.d.
random variables by construction of the $*$-deep traps and shift invariance of the environment, we obtain
\begin{eqnarray}
  \e \left[  \ee^{- \beta_K  N^{-1/\alpha} H_{N' }}  \right] \le \e \left[
    \e_{\tau,|\delta_1^*-\nu}^{\delta_1^*} \left[ \ee^{- \beta_K  N^{-1/\alpha} T_{\delta_1^*}  }  \right]\right]^{\underline{\theta}_{N'}}+o(1).
\end{eqnarray}
Since an easy computation yields that $\p(\delta_1^*\neq \delta_1)=\p(\max_{0\le y \le \nu(N')} \tau_y \ge g(N'))=o((N' \varphi({N'}))^{-1})$ and  $\p(H_{-\nu(N')}<H_{\nu(N')})=o((N' \varphi(N'))^{-1})$ when $N' \to \infty$ (or equivalently when $N \to \infty$), we get
\begin{eqnarray}
\label{eq:upbound}
  \e \left[  \ee^{- \beta_K  N^{-1/\alpha} H_{N' }}  \right] \le \E^x \Big[e^{-\beta_K  N^{-1/\alpha} {T_x}} \vert \tau_x \ge g(N') \Big] ^{\underline{\theta}_{N'}}+o(1).
\end{eqnarray}
Now, using Lemma \ref{l:laplacetransform}, this yields
\begin{eqnarray}
\limsup_{N \to \infty} \e \left[  \ee^{- \beta_K  N^{-1/\alpha} H_{N' }}  \right] \le \exp\Big\{-  { \alpha \pi \over \sin( \alpha\pi)} \, v_\varepsilon^{-\alpha} \beta_K^\alpha(u_K-u_{K-1} )\Big\}.
\end{eqnarray}
Moreover, we can similarly obtain  the same lower bound, which
implies (\ref{eq:onedim}) and concludes the proof of the convergence of the finite-dimensional distributions.

For the tightness, the arguments are exactly the same as in \cite{benarous-bovier-cerny}. We refer to section $5$ of \cite{benarous-bovier-cerny} for a detailed discussion.
\end{proof}
\medskip
{\it Proof of Theorem  \ref{t:scaling}.}

\noindent We use $(D(\left[0,T\right]),M_1)$ (resp. $(D(\left[0,T\right]),U)$) to denote the space $D(\left[0,T\right])$ equipped with the $M_1$ (resp. uniform) topology.
Let us introduce
\begin{equation}
\overline X_t^{(N)}:= \sup_{0\le s\le t} X_s^{(N)},  \qquad t \ge 0,
\end{equation}
which corresponds to the generalized inverse of the increasing process $H^{(N)}.$ Let $D^{\uparrow}$ denote the subset of $D(\left[0,T\right])$ consisting of unbounded increasing functions. By corollary 13.6.4 of \cite{whitt} the inverse map from $(D^{\uparrow},M_1)$ to $(D^{\uparrow},U)$ is continuous at strictly increasing functions. Since the $\alpha$-stable subordinator $V_{\alpha}$
(which appears in the limit of $H^{(N)}$ in $(D^{\uparrow},M_1)$) is almost surely strictly increasing (indeed, its L\'evy measure, denoted by $\Pi_\alpha,$ satisfies $\Pi_\alpha((0,\infty))=\infty$), the distribution of $\overline X^{(N)}$ converges to the distribution of $v_\varepsilon^{\#} V_\alpha^{-1}$ weakly on $(D^{\uparrow},U)$ and the limit is almost surely continuous.
 
 Now, Theorem \ref{t:scaling} will be a consequence of
 \begin{equation}
 \label{eq:XoverlineX}
\p \left(\sup \left\{ \vert X^{(N)}_t - \overline X^{(N)}_t \vert; \, 0 \le t \le T \right\} > \gamma \right)  \longrightarrow 0, \qquad N \to \infty,
\end{equation}
for any $\gamma>0.$ To prove (\ref{eq:XoverlineX}), recall first that Proposition \ref{p:scaling limit} implies that $\p(H_{N^\alpha\log N} >TN )\to 1,$ when $ N\to \infty,$ such that we only have to prove
 \begin{equation}
 \label{eq:XoverlineX++}
\p \left(\sup\{\vert X_t - \overline X_t \vert ; \, 0 \le t \le  H_{\lfloor N^\alpha\log N \rfloor}\} > \gamma N^\alpha \right)  \longrightarrow 0, \qquad N \to \infty.
\end{equation}
Furthermore, observe that  
 \begin{equation}
\sup\{\vert X_t - \overline X_t \vert ; \, 0 \le t \le  H_{\lfloor N^\alpha\log N \rfloor}\}=\max\{\vert Y_k - \overline Y_k \vert ; \, 0 \le k \le  \zeta_{\lfloor N^\alpha\log N \rfloor}\},
\end{equation}
by definition and that on $\mathcal{A}(\lfloor N^\alpha\log N \rfloor)$ (whose probability tends to $1$ when $N$ goes to infinity), this last quantity is less than $\nu(\lfloor N^\alpha\log N \rfloor)=o(N^{\alpha}),$ when $N \to \infty.$ This yields (\ref{eq:XoverlineX++}) and concludes the proof of Theorem  \ref{t:scaling}.
\qed

\section{Proof of Theorem \ref{t:aging} } \label{s:aging}
To bound the number of traps the random walk can cross before time
$t$ let us consider
  \begin{equation}
n_t:= \lfloor  t^\alpha \log \log t\rfloor ,
\end{equation}
 and observe that Theorem \ref{t:scaling} implies that $\p(\overline X_t \ge n_t) \to 0,$ $t \to \infty.$ Moreover, since we need more concentration properties for the random walk in the neighborhood of the $\delta_j$'s, we introduce  
  \begin{equation}
\overline \nu= \overline \nu(n_t):=  \lfloor  C' \log \log n_t \rfloor ,
\end{equation}
for some $C'$ large enough which will be chosen later. For convenience of notations we will use $\nu,$  $\overline \nu$ and $\delta_j$ in place of $\nu(n_t),$ $\overline \nu(n_t)$ and $\delta_j(n_t)$ throughout this section.
 
 Then, we define the sequence of random times $(T^*_j)_{j \ge 1}$ as follows: conditioning on $\tau,$ $(T^*_j)_{j \ge 1}$ is defined as an independent sequence of random variables with the law of $H_{\delta_j^*+\overline \nu}$ in the environment $\tau$ starting at site $\delta_j^*$ and reflected at $\delta_j^*- \nu.$ Hence, under the annealed law $\p,$ the $T^*_j$'s are are i.i.d. since the intervals $B_{\nu}(\delta_j^*)$ are made of independent and identically distributed portions of environment $\tau$ (by definition). Then, we give an analogous result to the extension of Dynkin's theorem proved in \cite{enriquez-sabot-zindy-3} (see Proposition 1 in \cite{enriquez-sabot-zindy-3}).
\begin{proposition}\label{p:dynkin}
For any $t > 0,$ let
$\l_t^*:=\sup\{j \ge 0 : \; T_1^*+\cdots +T_j^*\le t\}.$ Then, for all $0\le x_1<x_2 \le 1,$ we have
\begin{equation}
\lim_{t\to\infty} \p(t(1-x_2)\le T_1^*+\cdots +T_{\l_t^*}^*\le
t(1-x_1)) ={\sin(\alpha\pi)\over \pi} \int_{x_1}^{x_2}
{ (1-x)^{\alpha -1} x^{-\alpha}} \d x.
\end{equation}
For all $0\le x_1<x_2$, we have
\begin{equation}
\lim_{t\to\infty} \p(t(1+x_1)\le T_1^*+\cdots
+T_{\l_t^*+1}^*\le t(1+x_2)) ={\sin(\alpha\pi)\over \pi}
\int_{x_1}^{x_2} {\d x\over x^{\alpha} (1+x)}.
\end{equation}
\end{proposition}
Before proving this result, let us first recall Lemma \ref{l:laplacetransform} and make the following observation, which is the main ingredient in the proof of Proposition \ref{p:dynkin}.
\begin{remark}
\label{r:laplacetransform+} If we consider
\begin{eqnarray}
T^*(x)=T^*(x,n_t)&:=&\sum_{0}^{\zeta_{x+\overline \nu(n_t)}} \tau_{Y_i} {\bf e}_i{\bf 1}_{\{{Y_i}\in \left[x-\nu(n_t), \,x+\overline \nu(n_t)\right]\}}, \qquad x \in \z,
\end{eqnarray}
then the same arguments as in the proof of Lemma \ref{l:laplacetransform} yield, for all $\lambda>0,$
 \begin{equation}
 \E^x \Big[1-e^{-\lambda {T^*(x) \over t}} \vert \tau_x \ge g(n_t) \Big] \sim
{\p(\tau_x \ge g(n_t))^{-1} \over t^\alpha} { \alpha \pi \over \sin( \alpha\pi)} \, v_\varepsilon^{-\alpha}
\, \lambda^\alpha, \qquad t \to \infty.
\end{equation}
\end{remark}

\bigskip

\begin{proof} Observe first that an easy computation yields that $\p^x(H_{x- \nu}<\infty)=O(r_\varepsilon^{ \nu}),$ when $t \to \infty$ (where we recall that $ r_\varepsilon= q_\varepsilon/ p_\varepsilon<1$).  Moreover, we have $r_\varepsilon^{ \nu(n_t)}=o((t^\alpha \varphi(n_t))^{-1}).$ Therefore, Remark \ref{r:laplacetransform+} yields
 \begin{equation}
 \E \Big[1- \ee^{-\lambda {T^*_1 \over t}} \Big] \sim
{\p(\tau_x \ge g(n_t))^{-1} \over t^\alpha} { \alpha \pi \over \sin( \alpha\pi)} \, v_\varepsilon^{-\alpha}
\, \lambda^\alpha, \qquad t \to \infty.
\end{equation}
Then, the arguments are exactly the same as in the proof of Proposition $1$ in \cite{enriquez-sabot-zindy-3}. Observe that this result would exactly be Dynkin's theorem (see Feller, vol. II, \cite{feller}, p. 472) if the sequence $(T^*_j)_{j \ge 1}$ was an independent sequence of random variables in the domain of attraction of a stable law of index $\alpha.$ Here, this sequence depends implicitly  on the time $t,$ since the $*$-deep traps are defined from the critical depth $g(n_t).$
\end{proof}

Recalling Lemma \ref{l:IA}, we will now prove that the results of Proposition \ref{p:dynkin} are still true if we consider, in addition, the inter-arrival times between deep traps. Before, let us define the notion of inter-arrival times between $x$ and $y$, for any $x,y \ge 0,$ by:
\begin{equation}
H(x,y):=\inf \{ t \ge 0: \;   X_{H_x+t}=y \}.
\end{equation}

\begin{proposition}\label{p:dynkin2}
For any $t > 0,$ let
$\l_t:=\sup\{j \ge 0 : \; H_{\delta_j}\le t\}.$ Then, we have
\begin{equation}
\lim_{t\to\infty} \p(H_{\delta_{\ell_t}}\le t <H_{\delta_{\ell_t}+\overline \nu}  )=1.
\end{equation}
For all $0\le x_1<x_2 \le 1,$ we have
\begin{equation}
\lim_{t\to\infty} \p(t(1-x_2)\le H_{\delta_{\ell_t}} \le
t(1-x_1)) ={\sin(\alpha\pi)\over \pi} \int_{x_1}^{x_2}
{ (1-x)^{\alpha -1} x^{-\alpha}} \d x.
\end{equation}
For all $0\le x_1<x_2$, we have
\begin{equation}
\lim_{t\to\infty} \p(t(1+x_1)\le H_{\delta_{\ell_t+1}}\le t(1+x_2)) ={\sin(\alpha\pi)\over \pi}
\int_{x_1}^{x_2} {\d x\over x^{\alpha} (1+x)}.
\end{equation}
\end{proposition}
\begin{proof}
We first need to prove that after hitting $\delta_j+\overline \nu,$ the particle does not backtrack more than $\overline \nu.$ We detail this result with the following lemma.
\begin{lemma}
\label{l:DT+} Let us define $\mathcal{B}(n_t):=\mathcal{A}(n_t) \cap \bigcap_{j=1}^{\theta_{n_t}}  \{  H(\delta_j+\overline \nu,\delta_j+ \nu) < H(\delta_j+\overline \nu,\delta_j)   \}. $ Then, we have
  \begin{equation}
\lim_{t \to \infty}\p\left(\mathcal{B}(n_t)\right)=1.
\end{equation}
\end{lemma}
\begin{proof}
Since Lemma \ref{f:preliminaries-traj} says that $\p\left(\mathcal{A}(n_t)\right)$ tends to one, we only have to prove that 
  \begin{equation}
\lim_{t \to \infty}\p\Big( \bigcup_{j=1}^{\theta_{n_t}}  \{  H(\delta_j+\overline \nu,\delta_j+ \nu) > H(\delta_j+\overline \nu,\delta_j)   \}\Big)=0.
\end{equation}
Recalling that on $\mathcal{E}(n_t)\cap\mathcal{E}^*(n_t),$ whose probability tends to 1 when $t$ tends to infinity (by Lemma \ref{l:preliminaries} and Lemma \ref{l:*}), the number $\theta_{n_t}$ of deep traps (i.e. deeper than $g(n_t)$) is bounded by $C (\log n_t)^{2 \alpha \over 1-\alpha},$ it is sufficient to prove that
  \begin{equation}
  \label{eq:limsum}
\lim_{t \to \infty} \left( \sum_{1\le j\le C (\log n_t)^{2 \alpha \over 1-\alpha}} \p \big( H(\delta_j+\overline \nu,\delta_j+ \nu) > H(\delta_j+\overline \nu,\delta_j) \big) \right)=0.
\end{equation}
Now, the strong Markov property applied at $H(\delta_j+\overline \nu)$ implies that the probability term in (\ref{eq:limsum}) is bounded by $\p(\zeta_{-\overline \nu} <\infty),$ which does not depend on $j.$  Therefore, (\ref{eq:limsum}) will be a consequence of 
  \begin{equation}
\p(\zeta_{-\overline \nu} <\infty)=o( (\log n_t)^{-{2 \alpha \over 1-\alpha}}), \qquad t \to \infty.
\end{equation}
Recalling that we have  $\p(\zeta_{-\overline \nu} <\infty) \le C r_\varepsilon^{\overline \nu}$  (where $ r_\varepsilon= q_\varepsilon/ p_\varepsilon<1$), we conclude the proof of Lemma \ref{l:DT+} by choosing $C'$ larger than $-2\alpha/(1-\alpha) \log r_\varepsilon$ (recall that $\overline \nu= \overline \nu(n_t)=  \lfloor C' \log \log n_t \rfloor$).
\end{proof}

Let us introduce $\mathcal{C}(n_t):=\{\overline {X}_t \le n_t\},$ whose probability tends to one (recall Theorem \ref{t:scaling}). Now, to prove Proposition \ref{p:dynkin2}, observe that
on $\mathcal{E}^*(n_t)\cap \mathcal{A}(n_t),$  the random
times $(H(\delta_j,\delta_j+ \overline \nu))_{1\le j \le \theta^*_{n_t}}$ have the same law as the random
times $(T^*_j)_{1\le j \le \theta^*_{n_t}}$ defined previously. If we define $\tilde \l_t :=\sup\{j \ge 0: \;
H(\delta_1,\delta_1+ \overline \nu)+\cdots +H(\delta_j,\delta_j+ \overline \nu) \le t\}$, then, using Proposition
\ref{p:dynkin}, Lemma \ref{l:*} and Lemma \ref{f:preliminaries-traj}, we get that the result of
Proposition \ref{p:dynkin} is true with $(H(\delta_j,\delta_j+ \overline \nu))_{1\le j \le \theta^*_{n_t}}$ and $\tilde \l_t$ in
place of $(T^*_j)_{1\le j \le \theta^*_{n_t}}$ and $\l^*_t$. Now, recalling Lemma \ref{l:IA} and since $n_t^{1/\alpha}/ \log n_t=o(t),$ when $t \to \infty,$ we obtain
\begin{eqnarray*}
&&\liminf_{t\to \infty} \p ( \tilde \l_t=\l_t -1\, ;\, H_{\delta_{\ell_t}}\le t <H_{\delta_{\ell_t}+\overline \nu} )
\\
&\ge& \liminf_{t\to\infty} \p( \mathcal{I}(n_t)\, ;\, \mathcal{B}(n_t)\, ;\, \mathcal{C}(n_t) \, ;\, \vert t- (H(\delta_1,\delta_1+ \overline \nu)+\cdots
+H(\delta_{\tilde \l_t},\delta_{\tilde \l_t}+ \overline \nu))\vert \ge \xi t ),
\end{eqnarray*}
for all $\xi>0.$ Thus, using Lemma \ref{l:IA}, Lemma \ref{l:DT+}, Proposition \ref{p:dynkin} (for
$\tilde \l_t$ and $(H(\delta_j,\delta_j+ \overline \nu))_{1\le j \le \theta^*_{n_t}})$ and letting $\xi$ tends to $0,$ we
get that
\begin{equation}
\lim_{t\to \infty} \p (\tilde \l_t=\l_t -1\, ;\, H_{\delta_{\ell_t}}\le t <H_{\delta_{\ell_t}+\overline \nu})=1.
\end{equation}
We conclude the proof by the same type of arguments.
\end{proof}

To complete the proof of Theorem \ref{t:aging}, we will prove the following {\it localization} result, which means that the particle is in the last visited deep trap with an overwhelming probability.
\begin{proposition}\label{p:localization}
We have
\begin{equation}
\lim_{t\to\infty} \p(X_t=\delta_{\ell_t})=1.
\end{equation}
\end{proposition}
\begin{proof}
Now, for any deep trap $\delta_j$, let us denote by $\mu_j$ the invariant measure associated with the trap model on $\left[ \delta_j-\nu, \delta_j+\overline \nu \right]$ reflected at sites $\delta_j-\nu$ and $\delta_j+\overline \nu$ and normalized such that $\mu_j(\delta_j)=1.$ Clearly, $\mu_j$ is the reversible measure given by
\begin{eqnarray}
\label{eq:invariant}
\mu_j(x)=r_{\varepsilon}^{\delta-x} {\tau_x  \over \tau_{\delta_j}}, \qquad x\in (\delta_j-\nu; \delta_j+\overline \nu) \cap \z.
\end{eqnarray}
Since the process is reflected at sites $\delta_j-\nu$ and $\delta_j+\overline \nu,$ we have $\mu_j(\delta_j-\nu) \le \tau_{\delta_j-\nu} /\tau_{\delta_j}$ and $\mu_j(\delta_j-\nu) \le r_{\varepsilon}^{\overline \nu} \tau_{\delta_j+\overline \nu} / \tau_{\delta_j}.$ Moreover, since $\mu_j$ is an invariant measure and since $\mu_j(\delta_j)=1$, we have, for any $x \in\left[\delta_j-\nu,\delta_j+\overline\nu\right]$ and all $s \ge 0,$
\begin{equation}
\label{eq:reflected}
\p^{\delta_j}_{\tau,|\delta_j-\nu,\delta_j+\overline\nu |}(X_s=x)\le{\mu_j}(x).
\end{equation}

Furthermore, let us introduce the event
\begin{equation}
\mathcal {D}(n_t) := \bigcap_{j=1}^{\theta_{n_t}}\Big\{ \max_{x \in B_{\nu}(\delta_j)\setminus\{\delta_j\} } \tau_x <  (\log n_t)^\beta  \Big\},
\end{equation}
with $\beta > {1 \over \alpha} ({2\alpha \over 1-\alpha } +1+\gamma).$ Observe that the probability of $\mathcal {D}(n_t)$ tends to one, when $t$ tends to infinity. Indeed, since the number of deep traps is less than $C (\log n_t)^{2 \alpha \over 1-\alpha},$ and recalling that the number of sites contained in the $B_{\nu}(\delta_j)$'s is less than $2 \nu$ (with $\nu=\nu(n_t)= \lfloor (\log n_t)^{1+\gamma} \rfloor $), this fact is just a consequence of (\ref{ass:stable}). Recalling (\ref{eq:invariant}), observe that on $\mathcal {D}(n_t)$ we have
\begin{equation}
\label{eq:invmajo}
\sup_{x \in \left[\delta_j-\nu,\delta_j+\overline\nu\right]\setminus\{\delta_j\} } {\mu_j}(x) \le C r_{\varepsilon}^{\overline \nu} {(\log n_t)^{\beta+{2\over 1-\alpha}} \,  n_t^{-{1  \over  \alpha}}} \le C n_t^{-{1  \over 2\alpha}},
\end{equation}
for any $1\le j \le \theta_{n_t}.$ Hence, combining (\ref{eq:reflected}) and (\ref{eq:invmajo}), we obtain on $\mathcal{D}(n_t)$
\begin{equation}
\label{eq:supprobmajo}
 \p^{\delta_j}_{\tau,|\delta_j-\nu,\delta_j+\overline\nu |}(X_s\neq \delta_j )\le C n_t^{-{1  \over 2\alpha}}, \qquad \forall \, s \ge 0.
\end{equation}

Now, we fix $0<\xi<1.$ Then, let us write that $\liminf_{t\to\infty} \p(X_t=\delta_{\ell_t})$ is larger than
\begin{eqnarray}
\nonumber
&& \liminf_{t\to\infty} \p(X_t=\delta_{\ell_t} \, ;\, \ell_{t}= \ell_{(1+\xi)t})
\\
&\ge&  \liminf_{t\to\infty} \p(\ell_{t}= \ell_{(1+\xi)t})- \limsup_{t\to\infty} \p(X_t \neq \delta_{\ell_t} \, ;\, \ell_{t}= \ell_{(1+\xi)t}).
\label{eq:conclu1}
\end{eqnarray}
Considering the first probability term in (\ref{eq:conclu1}), we get using Proposition
\ref{p:dynkin2} that it is equal to
\begin{eqnarray}
\liminf_{t\to\infty} \p (H_{\delta_{\ell_t+1}}>(1+\xi)t)
= {\sin(\alpha \pi)\over \pi} \int_{\xi}^\infty
  {\d x\over x^{\alpha}
(1+x)}. \label{eq:conclu2}
\end{eqnarray}
In order to estimate the second probability term in (\ref{eq:conclu1}), let us introduce the event
\begin{equation*}
\mathcal {F}(n_t) :=  \mathcal {B}(n_t) \cap   \mathcal {C}(n_t) \cap  \mathcal {D}(n_t) \cap  \mathcal {E}(n_t) \cap  \mathcal {E}^*(n_t) \cap   \mathcal {I}(n_t) \cap \left\{ H_{\delta_{\ell_t}}\le t <H_{\delta_{\ell_t}+\overline \nu}  \right\}.
\end{equation*}
Observe
that the preliminary results obtained in Section
\ref{s:prelim} together with Theorem \ref{t:scaling}, Proposition \ref{p:dynkin2} and Lemma \ref{l:DT+}
imply that $\p(\mathcal {F}(n_t) )\to 1,$ when $t \to \infty.$ Then, we have that $ \limsup_{t\to\infty} \p (X_{t}\neq \delta_{\ell_t}\, ;\, \l_{t}=\l_{t(1+\xi)})$ is less than
\begin{eqnarray}
&& \limsup_{t\to\infty} \p(\mathcal {F}(n_t)\,; \, X_{t}\neq \delta_{\ell_t}\, ;\, \l_{t}=\l_{t(1+\xi)})
 \\
&\le & \limsup_{t\to\infty} \e\Big[\indic_{\mathcal {F}(n_t)} \sum_{j=1}^{\theta_{n_t}}
\indic_{ \{X_{t}\neq \delta_{\ell_t}\, ;\, \l_{t}=\l_{t(1+\xi)}=j\}}\Big].
\nonumber
\end{eqnarray}
But on the event $\mathcal {F}(n_t) \cap \{\l_{t}=\l_{t(1+\xi)}=j\}$ we know that
for all $s \in [H_{\delta_j}, t]$ the walk $X_s$ is in the interval
$\left[ \delta_j-\nu, \delta_j+\overline \nu \right].$ Indeed, on the event $\mathcal {B}(n_t) \cap \mathcal {C}(n_t)\cap \mathcal {I}(n_t)$ we know that once the position $ \delta_j+\overline \nu $ is reached
then within a time $n_t^{1/\alpha}/ \log n_t=o(t),$ when $t \to \infty,$ the position $\delta_{j+1}$ is reached,
which would contradict the fact that $\l_{t(1+\xi)}=j$. Hence, we
obtain, for all $j\in \N,$
\begin{eqnarray}
&&\p\left(\mathcal {F}(n_t)\,;\, j \le \theta_{n_t}\,;\, X_{t}\neq \delta_{\ell_t}\, ;\, \l_{t}=\l_{t(1+\xi)}=j \right)
\\
& \le & \E\Big[ \indic_{\{j \le \theta_{n_t}\}}\indic_{\mathcal{D}(n_t)\cap \mathcal{E}(n_t)}
 \sup_{s\in [0,t]} \p^{\delta_j}_{\tau,|\delta_j-\nu,\delta_j+\overline\nu |}(X_s\neq \delta_{j}) \Big]
 \le C n_t^{-{1  \over 2\alpha}},
\nonumber
\end{eqnarray}
where we used  (\ref{eq:supprobmajo}) on the event $\mathcal{D}(n_t).$ Considering now that, on the event $\mathcal{E}(n_t),$ the number $\theta_{n_t}$
of deep traps is smaller than $C (\log n_t)^{{2 \alpha \over 1-\alpha}}$ we
get that
\begin{eqnarray}
\label{eq:conclu3}
 \limsup_{t\to\infty} \p (X_{t}\neq \delta_{\ell_t}\, ;\, \l_{t}=\l_{t(1+\xi)})=0.
\end{eqnarray}
Then, assembling (\ref{eq:conclu1}), (\ref{eq:conclu2}), (\ref{eq:conclu3}) and letting $\xi$ tends to $0$ in (\ref{eq:conclu2}) concludes the proof of Proposition \ref{p:localization}.
\end{proof}

\medskip
{\it Proof of Theorem \ref{t:aging}.}
let us fix $h>1$ and introduce the event
\begin{eqnarray}
\mathcal {G}(t,h) :=  \{X_{t}= \delta_{\ell_t}\} \cap
\{X_{th}=\delta_{\ell_{th}}\},
\end{eqnarray}
whose probability tends to 1, when $t$ tends to infinity (it is a
consequence of Proposition \ref{p:localization}). Then, we easily have $\{X_{th}=X_t\} \cap\mathcal {G}(t,h)=\{ \l_{th}=\l_t\} \cap  \mathcal {G}(t,h).$
Therefore, since Proposition \ref{p:dynkin2} implies
that $\lim_{t\to\infty} \p (\l_{th}=\l_t)$ exists, we obtain
\begin{eqnarray}
\lim_{t\to\infty} \p(X_{th} =X_t) &=
& \lim_{t\to\infty} \p (\l_{th}=\l_t)= \lim_{t\to\infty} \p( T_{\l_t+1}\ge th)
\\
&=&  {\sin(\alpha\pi)\over \pi} \int_{0}^{1/h} y^{\kappa-1} (1-y)^{-\kappa}\d y,
\nonumber
\end{eqnarray}
which concludes the proof of Theorem \ref{t:aging}.
\qed

\bigskip
\bigskip
\noindent {\bf Acknowledgements}
 Many thanks are due to an anonymous referee for
careful reading of the original manuscript and for invaluable
suggestions.

\end{document}